\documentclass{birkmult}
\usepackage{graphicx,amssymb}

\makeatletter

\theoremstyle{plain}
\newtheorem {theorem}{Theorem}
\newtheorem {lemma}{Lemma}
\newtheorem {corollary}{Corollary}
\newtheorem*{conjecture}{Conjecture}
\newtheorem {proposition}{Proposition}

\theoremstyle{definition}
\newtheorem {definition}{Definition}

\newtheorem {remark}{Remark}

\let\@newpf\proof \let\proof\relax 
\newenvironment{proof}{\@newpf[\proofname]}{\qed\endtrivlist}

\DeclareMathOperator{\ch}{ch}
\DeclareMathOperator{\Res}{Res}
\DeclareMathOperator{\Ind}{Ind}
\DeclareMathOperator{\Hom}{Hom}
\DeclareMathOperator{\PF}{PF}

\def\hm#1{#1\nobreak\discretionary{}{\hbox{\m@th$#1$}}{}}

\makeatother

\DeclareMathAlphabet{\mathbbold}{U}{bbold}{m}{n}

\DeclareSymbolFont{rsfscript}{OMS}{rsfs}{m}{n}
\DeclareSymbolFontAlphabet{\mathrsfs}{rsfscript}

\DeclareFontFamily{OMS}{rsfs}{\skewchar\font'177}
\DeclareFontShape{OMS}{rsfs}{m}{n}{%
      <5> rsfs5
      <6> <7> rsfs7
      <8> <9> <10> rsfs10
      <10.95> <12> <14.4> <17.28> <20.74> <24.88> rsfs10
      }{}

\def\calF{\mathrsfs{F}}

\def\calL{\mathrsfs{L}}

\def\frA{\mathfrak{A}}
\def\frB{\mathfrak{B}}
\def\frC{\mathfrak{C}}

\def\frH{\mathfrak{H}}

\def\frM{\mathfrak{M}}
\def\frN{\mathfrak{N}}

\def\frP{\mathfrak{P}}
\def\frQ{\mathfrak{Q}}

\def\frU{\mathfrak{U}}
\def\frV{\mathfrak{V}}

\begin{document}
%
%
%
%
%
%
%
%

\title{Parking functions and vertex operators}
\author{Vladimir Dotsenko}
\address{%
Dublin Institute for Advanced Studies,\\
10 Burlington Road,\\
Dublin 4,\\
Ireland}
\email{vdots@maths.tcd.ie}

\thanks{The author was partially supported of the CNRS--RFBR grant no.~07-01-92214, by the President of the Russia grant~NSh-3472.2008.2, and by an IRCSET research fellowship.}

\begin{abstract}
We introduce several associative algebras and families of vector spaces associated to these algebras. Using lattice vertex operators, we obtain dimension and character formulae for these spaces. In particular, we define a family of representations of symmetric groups which turn out to be isomorphic to parking function modules. We also construct families of vector spaces whose dimensions are Catalan numbers and Fuss--Catalan numbers respectively. Conjecturally, these spaces are related to spaces of global sections of  vector bundles on (zero fibres of) Hilbert schemes and representations of rational Cherednik algebras.
\end{abstract}

\maketitle

\section{Introduction}

\subsection{Description of results} In his celebrated works~\cite{Ha1,Ha2} on the ``$n!$ conjecture'', Mark Haiman introduced several collections of bi-graded representations of symmetric groups whose dimensions are equal to numbers of parking functions, Catalan numbers, higher Catalan numbers etc. These spaces appear as global sections of certain sheaves on the zero fibre of the Hilbert scheme of points in the plane. Since then, there were several results that shed light on the representation--theoretical explanation of Haiman's result. In particular, Iain Gordon and Toby Stafford in~\cite{GS} related Haiman's spaces to finite-dimensional representations of rational Cherednik algebras (for special values of parameters). In this paper, we study some noncommutative algebras and associate to them some families of finite-dimensional vector spaces of the same dimensions as Haiman's spaces. We expect that these spaces are related to those of Haiman and Gordon--Stafford. 

Denote by~$\frB(k)$ the associative algebra with generators $e_p[i]$ ($p=1,\ldots,k$, $i\in\mathbb{Z}$) and relations 
 $$
\{e_p[i],e_q[j]\}=\{e_p[i+1],e_q[j-1]\} \text{ for all } i,j, p,q.
 $$ 
In other words, for every fixed distinct $p,q$ the anti-commutator 
 $$
\{e_p[i],e_q[j]\}:=e_p[i]e_q[j]+e_q[j]e_p[i]
 $$ 
depends only on the sum~$i+j$. We shall refer to subscripts as ``colours'': the algebra $\frB(k)$ has generators of $k$ different colours.

Consider the right $\frB(k)$-module $\frN$ induced from the trivial representation of the subalgebra $\frB_+(k)$ generated by $e_p[i]$ for all $p=1,\ldots,k$ and all $i>0$: $\frN=\mathbb{C}v\otimes_{\frB_+}\frB$, where $ve_p[i]=0$ for $e_p[i]\in\frB_+$. Let~$\frB_{\ge0}(k)\hm=\mathbb{C}\{ e_p[i]\mid i\ge0\}$. Denote by~$\frQ_{n}(k)$ be the image of the mapping from the $k^\text{th}$ tensor power $\frB_{\ge0}(k)^{\otimes n}$ to~$\frN$ that maps 
each tensor product 
 $$
e_{p_1}[i_1]\otimes e_{p_2}[i_2]\otimes\ldots\otimes e_{p_n}[i_n]
 $$ 
to 
 $$
ve_{p_1}[i_1]e_{p_2}[i_2]\ldots e_{p_n}[i_n].
 $$ 
In addition, consider the subspace $\frQ_{n}^{multi}(n)$ of $\frQ_{n}(n)$ spanned by all monomials containing generators of all colours (so that each colour occurs exactly once). 

\begin{theorem}\label{ParkingCor1}
\begin{enumerate}
 \item The dimension of the subspace $\frQ_{n}^{multi}(n)$ is equal to the number of parking functions of length~$n$, that is~$(n+1)^{n-1}$.
 \item Moreover, the subspace $\frQ_{n}^{multi}(n)$ equipped with the natural action of the symmetric group~$S_n$ (which permutes the colours) is isomorphic to the parking functions module~$\PF(n)$.
\end{enumerate}
\end{theorem}

The following result relates all vector spaces~$\frQ_n(k)$ (not only the ``multilinear components''~$\frQ_{n}^{multi}(n)$) to parking functions.

\begin{theorem}\label{ParkingCor2}
The vector space $\frQ_n(k)$ has a natural $\mathbb{Z}^k$-grading (that counts colours); consider the generating function $\ch_n(x_1,\ldots,x_k)$ of dimensions of its homogeneous components;
 $$
\ch_n(x_1,\ldots,x_k)=\sum_{i_1,\ldots,i_k}\dim(\frQ_n(k))_{i_1,\ldots,i_k}x_1^{i_1}\cdot\ldots\cdot x_k^{i_k}.
 $$
Then  
 $$
\ch_n(x_1,\ldots,x_k)=\mathcal{F}_{\PF(n)}(x_1,\ldots,x_k),
 $$
where $\mathcal{F}_{\PF(n)}(x_1,\ldots,x_k)$ is the Frobenius character of the parking functions module (more precisely, its projection from symmetric functions to symmetric polynomials in $k$ variables).
\end{theorem}

\subsection{Outline of the paper} The main idea behind the proofs of all our statements is to use some particular lattice vertex algebras. More precisely, the paper is organised as follows.
 
In Section~\ref{technique}, we recall some basic facts on lattice vertex operators for infinite-dimensional Heisenberg algebras. We present a self-contained introduction to lattice vertex operators; we restrict ourselves to the basic results required for our proofs. There are several articles and books one can refer to in order to get a whole variety of viewpoints, for example \cite{Dong,FK,Kac,LW,LX,Ro}. In this paper, we prefer a somewhat different approach and try to follow the philosophy of the paper of Feigin and Fuchs~\cite{FF}, who used vertex operators to construct singular vectors in Verma modules for the Virasoro algebra. The underlying vector space in Feigin--Fuchs approach is the space of semiinfinite forms, and the Heisenberg algebra appears naturally as a consequence of the so called boson--fermion correspondence. Our  Proposition~\ref{gen-span} below should be thought of as some version of the boson--fermion correspondence for operators which are slightly more complicated than the fermions. The main case of interest for our paper is a very degenerate case of a lattice equipped with a bilinear form whose kernel has codimension~$1$.
 
In Section~\ref{parking}, we prove the results announced above. During the preparation of this paper, we learned that the monomial bases which appear as the key ingredient of our proof were introduced previously in a paper of Roitman~\cite{Ro}. However, our proofs are essentially different: instead of using the machinery of rewriting and the Diamond lemma, we use explicit representations or our algebras which help to get lower bounds on dimensions and characters (thus making explicit a remark of Dong quoted in~\cite{Ro}).
 
We conclude with some important particular cases and some generalisations of our results, and some open questions. In particular, we discuss the possible relation between our vector spaces and vector spaces of Haiman and Gordon--Stafford.

All vector spaces and algebras in this paper are defined over the field of complex numbers.

\subsection{Acknowledgements} The author wishes to thank his teacher Boris Feigin who originally conjectured Theorems \ref{ParkingCor1} and \ref{CatalanCor} of this paper. Feigin's seminar where he generously shares his ideas with students and colleagues is probably the only thing that the author misses very much from Moscow. The author is also grateful to Evgeny Sklyanin for his questions on Catalan numbers; his interest in this work was really stimulating. 

\section{Lattice vertex operators: a summary}\label{technique}

\subsection{General statements}

\begin{definition}
Let $H$ be a lattice, that is, a free abelian group equipped with an integer-valued bilinear form $(\cdot,\cdot)$. Denote by $\frH$ the complex vector space~$\mathbb{C}\otimes_{\mathbb{Z}}H$ spanned by our lattice. The lattice Heisenberg Lie algebra corresponding to this data is the vector space     
 $$
\widehat{\frH}=\frH\otimes\mathbb{C}[t,t^{-1}]\oplus\mathbb{C}K
 $$ 
with the Lie bracket
\begin{gather*}
[h_1\otimes p(t),h_2\otimes q(t)]=(h_1,h_2) \Res\limits_{t=0}p(t)dq(t) K,\\
[K,h\otimes p(t)]=0.
\end{gather*}
\end{definition}

Throughout this paper, all modules over all algebras we study are considered to be \emph{right} modules. We also adopt the notation $h[i]:=h\otimes t^i$ for every $h\in\frH$.

\begin{definition}
Let $h\in\frH$. \emph{The Fock space representation $\calF_h$ of the Heisenberg algebra} is, by definition, the representation induced from a 1-dimensional representation $T_h$ of the subalgebra $\frH\otimes\mathbb{C}[t]\oplus\mathbb{C}K\subset\widehat{\frH}$, where $K$ acts by $1$, and $h'\otimes p(t)$ acts by the scalar~$(h,h')p(0)$. The basis element of the distinguished copy of $T_h$ inside the induced representation is called the \emph{vacuum vector} of the Fock space and is denoted by $v_h$.
\end{definition}

\begin{definition}
Let $h',h\in\frH$. A (lattice) vertex operator is a formal expression
 $$
V_{h',h}(z)=\sum\limits_{k\in\mathbb{Z}}V_{h',h}[k]z^{-k}
 $$
where the coefficients $V_{h',h}[i]$ (components of the vertex operator) are linear operators, 
 $$
V_{h',h}[i]\colon\calF_{h'}\to\calF_{h'+h}\ \  (i\in\mathbb{Z})
 $$ 
satisfying the following conditions:
\begin{itemize}
 \item Commutator relations between our operators and the operators of the Heisenberg algebra, namely
 $$
[h''[j], V_{h',h}[i]]=-(h,h'')V_{h',h}[j+i] \text{ for all } i,j;
 $$
these relations are equivalent to the following commutator relations for the vertex operator: 
 $$
[h''\otimes p(t), V_{h',h}(z)]=-(h,h'')p(z)V_{h',h}(z).
 $$
 \item Homogeneity: $V_{h',h}[i]$ is a homogeneous operator of degree $i$ (for natural gradings on the Fock spaces induced from the grading on the Heisenberg algebra for which $\deg(h[j])=j$). 
 \item Normalisation: $$v_{h'}V_{h',h}[0]=v_{h'+h}.$$
\end{itemize}
\end{definition}

\begin{proposition}\label{exist}
Let $D_{h',h}\colon \calF_{h'}\to \calF_{h'+h}$ be the operator uniquely determined from the following conditions:
  \begin{itemize}
   \item[(i)] $v_{h'}D_{h',h}=v_{h'+h}$;
   \item[(ii)] $[h''[j],D_{h',h}]=0$ for $j\ne0$.
  \end{itemize}
Define the formal generating function $B^{h',h}(z)$ by the following formula:
 $$
B^{h',h}(z)=\exp\left(\sum_{i<0} h[-i]\frac{z^{i}}{i}\right)
D_{h',h}
\exp\left(\sum_{i>0}h[-i]\frac{z^i}{i}\right).
 $$
(Coefficients of $z^j$ in this generating function are well defined, since $\calF_{h'}$ has no elements of positive degree.) Then components of this generating function give a vertex operator.
\end{proposition}

\begin{proof}
Homogeneity is obvious; the normalisation property is also quite straightforward, if one takes into account the fact that all the operators $h[j]$ for $j>0$ annihilate the vacuum vector. To prove the commutator relations, note that for each $j$ the element $h''[j]$ commutes with two of three factors in the formula for $B^{h,h'}(z)$. Let us consider the remaining factor in each case. For $j<0$ we have
\begin{multline*}
\left[h''[j], \exp\left(\sum_{i<0} h[-i]\frac{z^{i}}{i}\right)\right]=\\=
\left[h''[j], \sum_{i<0}h[-i]\frac{z^{i}}{i}\right]\exp\left(\sum_{i<0} h[-i]\frac{z^{i}}{i}\right)=\\=
-(h'',h)z^j \exp\left(\sum_{i<0} h[-i]\frac{z^{i}}{i}\right).
\end{multline*}
In the case $j>0$,
\begin{multline*}
\left[h''[j], \exp\left(\sum_{i>0} h[-i]\frac{z^{i}}{i}\right)\right]=\\=
\left[h''[j], \sum_{i>0}h[-i]\frac{z^{i}}{i}\right]\exp\left(\sum_{i>0} h[-i]\frac{z^{i}}{i}\right)=\\=
-(h'',h)z^j \exp\left(\sum_{i>0} h[-i]\frac{z^{i}}{i}\right).
\end{multline*}
Finally, for $j=0$ to compute the commutator $[h'',D_{h,h'}]$ it is enough to compute it on the vacuum vector (since both operators commute with the operators $h[i]$ for $i\ne0$, and the Fock space is generated from the vacuum vector by these operators):
\begin{multline*}
v_h[h'',D_{h,h'}]=v_hh''D_{h,h'}-v_hD_{h,h'}h''=\\=(h'',h')v_{h'+h}-(h'',h'+h)v_{h'+h}=-(h'',h)v_{h'+h},
\end{multline*}
and our proposition follows.
\end{proof}

From now on by vertex operators we mean those constructed in this proposition. (One can prove that in the case of a nondegenerate bilinear form $(\cdot,\cdot)$ there are no other vertex operators.)

\begin{proposition}\label{rel}
Vertex operators satisfy the following quadratic relations\footnote{These are typically not relations in the usual sense, since they involve infinite sums. Still, in all our representations all but a finite number of summands annihilate each vector, so both sides are well defined operators and the statement makes sense. Nevertheless, for the cases of particular interest to us the sums will be finite.}
\begin{multline*}
\nu_0 V_{h',h_1}[i] V_{h'+h_1,h_2}[j]+\nu_1 V_{h',h_1}[i+1] V_{h'+h_1,h_2}[j-1]+\ldots=\\
=\nu_0 V_{h',h_2}[j] V_{h'+h_2,h_1}[i]+\nu_1 V_{h',h_2}[j+1] V_{h'+h_2,h_1}[i-1]+\ldots,
\end{multline*}
where $\sum\limits_{i\ge0}\nu_lt^l=(1-t)^{-(h_1,h_2)}$.
\end{proposition}

\begin{proof}
The key ingredient of the proof is the following classical statement from linear algebra (which is also a special case of the Baker--Campbell--Hausdorff formula from the theory of Lie algebras):
\begin{lemma}
If $U$ and $V$ are linear operators for which $[U,V]$ commutes with~$U$ and $V$, then
 $$
e^U e^V=e^{[U,V]}e^V e^U.
 $$ 
\end{lemma}
To prove the formula, one should consider the exponential formula for the product $V_{h',h_1}(z)V_{h'+h_1,h_2}(w)$ and apply the previous lemma to exchange two factors: the factor 
 $$
\exp\left(\sum_{i>0}h_1[-i]\frac{z^{i}}{i}\right)
 $$
of $V_{h',h_1}(z)$ and the factor
 $$
\exp\left(\sum_{i<0}h_2[-i]\frac{w^{i}}{i}\right)
 $$
of $V_{h'+h_1,h_2}(w)$. The ``cost'' of this exchange is the factor 
 $$
\exp\left(\left[\sum_{i>0}h_1[-i]\frac{z^{i}}{i}, \sum_{i<0}h_2[-i]\frac{w^{i}}{i}\right]\right)
 $$
which is, as it is easy to check, equal to
 $$
\left(1-\frac{z}{w}\right)^{(h_1,h_2)}.
 $$
The same can be done for the product $V_{h',h_2}(w)V_{h'+h_2,h_1}(z)$, and as a result we get
\begin{multline*}
\left(1-\frac{z}{w}\right)^{-(h_1,h_2)}V_{h',h_1}(z)V_{h'+h_1,h_2}(w)=\\=\left(1-\frac{w}{z}\right)^{-(h_1,h_2)}V_{h',h_2}(w)V_{h'+h_2,h_1}(z)
\end{multline*}
(note that the $[h_1[i],h_2[j]]=0$ for $i$ and $j$ of the same sign), which, after formal expansion, gives precisely the relations that we claimed.
\end{proof}

\begin{definition}
Fix a basis $f_1,\ldots,f_k$ of $H$. Denote by $\frV$ the vector space $\bigoplus_{h\in H}\calF_h$. For $p=1,\ldots,k$ let us define operators $U_p[i]$ on this direct sum as follows: 
 $$
\left.U_p[i]\right|_{F_{h}}:=V_{h,f_p}[i+(h,f_p)].
 $$ 
\end{definition}

The following proposition is obvious.

\begin{proposition}\label{simp-rel}
Quadratic relations from Proposition~\ref{rel} can be rewritten as
\begin{multline*}
\nu_0 U_p[i] U_q[j-N]+\nu_1 U_p[i+1] U_q[j-N-1]+\ldots=\\
=\nu_0 U_q[j] U_p[i-N]+\nu_1 U_q[j+1] U_p[i-N-1]+\ldots,
\end{multline*}
where $\sum\limits_{i\ge0}\nu_lt^l=(1-t)^{-N}$, $N=(f_p,f_q)$.
\end{proposition}

\begin{definition}
Let $h\in H$. Denote by $\frV_h$ the minimal linear subspace of $\frV$ containing~$v_{h}\in \calF_h\subset\frV$ and stable under the action of all the~$U_p[i]$. Since $$v_{h} U_p[-(h,f_p)]=v_{h+f_p},$$ it is clear that~$\frV_{h+f_p}\subset \frV_{h}$. Denote by $\overline{\frV}$ the union of these subspaces.
\end{definition}

\begin{proposition}\label{gen-span}
The subspace $\overline{\frV}$ coincides with~$\frV$. In other words, not only the vertex operators can be defined in terms of the Heisenberg algebra, but also the Heisenberg algebra operators can be recovered from vertex operators, if we know all vacuum states~$v_h$.
\end{proposition}

\begin{proof}
Define the operators $B^{h',h,-}[i]\colon\calF_{h'}\to\calF_{h'}$ as follows:
 $$
\sum_{k\in\mathbb{Z}}B^{h',h,-}[i]z^{-i}=\exp\left(\sum_{j>0} h[-j]\frac{z^{j}}{j}\right).
 $$ 
Denote by $D_p$ the operator on~$\frV$ which is defined on each component~$\calF_{h}$ as the corresponding operator~$D_{h,f_p}$ from Proposition~\ref{exist}.

It is easy to see that the space~$\overline{\frV}$ is stable under the action of the operators~$h[j]$ for~$h\in H, j\ge0$, and~$D_p$ for $p=1,\ldots,k$. By the definition, it is stable under the action of the operators~$U_p[i]$. Thus, from the explicit formula for the vertex operator, we see that the space~$\overline{\frV}\cap\calF_{h'}$ is stable under the action of the operators~$B^{h',f_p,-}[i]$ for all~$p=1,\ldots,k$ and all~$i$. Taking the logarithm of their generating function, we obtain the generating function for the operators $f_p[i]$ with $i<0$ (and thus $h[i]$ for $h\in H$, $i<0$). This means that these operators also leave the subspace~$\overline{\frV}$ invariant. Thus we obtained a subspace of~$\frV$ containing all the vectors~$v_{l}$ which is stable under the action of the Heisenberg algebra. Such a subspace coincides with~$\frV$ for obvious reasons.
\end{proof}

\subsection{A particular case}

Further in this paper we are going to consider the $k$-dimensional lattice 
 $$B=\bigoplus_{p=1}^k\mathbb{Z}f_p$$
with $(f_p,f_q)=-1$. According to the previous section, we assign to this lattice the vector space~$\frV$ (the direct sum of the Fock representations).  

\begin{proposition}
The vector space~$\frV$ is naturally endowed with a structure of a representation of the algebra $\frB$ defined in the introduction.
\end{proposition}

\begin{proof}
We use the result of Proposition~\ref{simp-rel}: 
 $$
 U_p[i] U_q[j+1]- U_p[i+1] U_q[j]= U_q[j] U_p[i+1] - U_q[j+1] U_p[i]
 $$
(since $\sum\limits_{i\ge0}\nu_lt^l=1-t$, $N=-1$), which can be rewritten as
 $$
\{U_p[i],U_q[j+1]\}=\{U_p[i+1],U_q[j]\}, 
 $$
so the operators $U_p[i]$ satisfy the relations of the algebra~$\frB$.
\end{proof}

The remaining part of the paper is mostly devoted to the proof of the statements formulated in the introduction. For the proof, we use some kind of ``semiinfinite bases'' of our vector spaces, which then will be used to obtain bases of the vector spaces we are interested in.  

\section{The main theorem}\label{parking}

In this section, we only consider vertex operators associated the lattice~$B$, and we use the notation $e_p[i]$ for the operator $U_p[i]$. From the previous section we know that our vector space is spanned by all monomials 
 $$
v_h e_{p_1}[i_1]\ldots e_{p_s}[i_s]
 $$
for all $h\in H$, and all sequences $p,i$. Our next goal is to define a certain subset of all monomials of this type and then prove that our monomials form a basis of~$\frV$.

\subsection{Admissible monomials}
\begin{definition}\label{park-usl}
We call a monomial
 $$
e_{h,\mathbf{p},\mathbf{i}}=v_h e_{p_1}[i_1]\ldots e_{p_s}[i_s]
 $$
\emph{admissible}, if the following conditions are satisfied:
\begin{enumerate}
 \item $i_1\le-(h,f_{p_1})$;
 \item $i_{m+1}\le i_m+1$ for all $m=1,\ldots,s-1$;
 \item if $i_{m+1}=i_m+1$, then $p_m\le p_{m+1}$.
\end{enumerate}
Clearly, 
 $$
v_h e_{p_1}[-(h,f_1)]e_{p_2}[i_2]\ldots e_{p_s}[i_s]=v_{h+f_1} e_{p_2}[i_2]\ldots e_{p_s}[i_s],
 $$
and in this case we think of these monomials as of the same monomial. Throughout this paper we mean by an admissible monomial the class of pairwise equal admissible monomials; we hope it will not lead to a confusion.
\end{definition}

\begin{remark}
There is another way to think about the latter definition of equivalence classes of monomials. Let us discuss it in the simplest case when our lattice has rank one. The operator~$e_1[lf_1]$ defines an embedding of vector spaces $\frV_{(l+1)f_1}\hookrightarrow\frV_{lf_1}$. Thus the space $\frV=\bigcup_{l}\frV_{lf_1}$ can be thought of as the linear span of ``semiinfinite monomials'' 
 $$
e_a=\ldots e_1[a_{-k}]e_1[a_{1-k}]\ldots e_1[a_{-2}]e_1[a_{-1}],
 $$ 
where the semiinfinite (infinite to the left) sequence of integers $a=\{a_i\}_{i\le-1}$ satisfies~$a_{i+1}=a_i+1$ for all sufficiently small $i$. The monomials that we consider to be equal are just different choices of the ``tail'' for the same semiinfinite monomial. 
\end{remark}

\begin{lemma}\label{park-span}
The vector space $\frV$ is spanned by all admissible monomials 
\end{lemma}

\begin{proof}
Note that all monomials without restrictions put on by the admissible monomials conditions span $\frV$ (see Proposition~\ref{gen-span}). It is clear that $v_h e_p[i]=0$ for $i>-(h,f_p)$, so the first condition is not restrictive. Thus it is enough to check that every vector~$v_h e_{p_1}[i_1]\ldots e_{p_s}[i_s]$ can be expressed as a linear combination of vectors of the same kind where the second admissible monomials condition is satisfied. Let us use induction on~$s$, and for the fixed $s$, induction on $i_s$. If~$s=1$, there is nothing to prove. Otherwise we use the induction hypothesis to find a corresponding linear combination for $ve_{p_1}[i_1]\ldots e_{p_{s-1}}[i_{s-1}]$, then multiply it by $e[i_s]$ and rewrite it using the relations, decreasing the parameter $i_s$. There are two possible situations that require some improvement. If $i_s\ge i_{s-1}+2$, then
\begin{multline*}
e_{p_{s-1}}[i_{s-1}]e_{p_s}[i_s]=e_{p_{s-1}}[i_{s-1}+1]e_{p_s}[i_s-1]+\\+e_{p_s}[i_s-1]e_{p_{s-1}}[i_{s-1}+1]-e_{p_s}[i_s]e_{p_{s-1}}[i_{s-1}],
\end{multline*}
where all the summands have the parameter $i_s$ less the original one. Otherwise, if $i_s=i_{s-1}+1$ and $p_{s-1}>p_{s}$, then
 $$
e_{p_{s-1}}[i_s-1]e_{p_s}[i_s]=e_{p_{s-1}}[i_s]e_{p_s}[i_s-1]+e_{p_s}[i_s-1]e_{p_{s-1}}[i_s]-e_{p_s}[i_s]e_{p_{s-1}}[i_s-1],
 $$ 
where the first and the third summand have the parameter $i_s$ less than the original one, and for the second summand the colour of the last factor is less than the original one, so for fixed $s$ and $i_s$ we can proceed by decreasing induction on that parameter. Now the induction hypothesis applies, and the lemma follows.
\end{proof}

Denote by~$\frM_{\infty/2}$ the vector space whose basis elements $m_{h,\mathbf{p},\mathbf{i}}$ are in one-to-one correspondence with admissible monomials. 
For such an element, we define a $\mathbb{Z}$-valued statistic $d_1(m_{h,\mathbf{p},\mathbf{i}})=i_1+\ldots+i_s-\frac12(h-f_1,h)$ and a $H$-valued statistic $d_2(m_{h,\mathbf{p},\mathbf{i}})=h+f_{p_1}+\ldots+f_{p_s}$. Note that the statistics $d_1$ and $d_2$ are well-defined, that is if we compute them for two elements that correspond to equal 
monomials
 $$
v_h e_{p_1}[-(h,f_1)]e_{p_2}[i_2]\ldots e_{p_s}[i_s]=v_{h+f_1} e_{p_2}[i_2]\ldots e_{p_s}[i_s],
 $$
the results will be the same.

Consider the group algebra of $H$ with coefficients in formal Laurent series in~$q$. It is spanned by formal symbols $z^h$, $h\in H$, that satisfy $z^{h_1+h_2}=z^{h_1}z^{h_2}$. For any $\mathbb{Z}\times H$-graded vector space $\mathfrak{L}=\bigoplus\limits_{i\in\mathbb{Z},\,h\in H}\mathfrak{L}_{i,h}$ denote by~$\ch\mathfrak{L}(q,z)$ its character given by the sum $\sum\limits_{i\in\mathbb{Z},\,h\in H}\dim\mathfrak{L}_{i,h}q^iz^h$.

\begin{definition}
Define a $\mathbb{Z}\times H$-grading on the vector space~$\frV$ as follows. On the space~$\calF_{h}\subset\frV$, we put $\deg_1(x):=\deg(x)-\frac12(h-f_1,h)$ and $\deg_2(x):=h$.
\end{definition}

\begin{lemma} 
We have
\begin{equation}\label{park-chV}
\ch\frV(q,z)=
\frac{1}{(q^{-1})_\infty^k}\sum_{h\in H}q^{-\frac12(h-f_1,h)}\,z^{h}.
\end{equation}
\end{lemma}

\begin{proof}
This lemma is straightforward from our definition, and the fact that as a vector space the Fock space is nothing but the space of polynomials in $h_p[i]=f_s[i]$, $i<0$, $1\le p\le k$.
\end{proof}

\begin{lemma}
We have
 $$
\ch\frM_{\infty/2}(q,z)=\ch\frV(q,z).
 $$
\end{lemma}

\begin{proof}
Consider an element $m_{h,\mathbf{p},\mathbf{i}}$ with $d_2(m_{h,\mathbf{p},\mathbf{i}})=d$. By the definition, this means that it corresponds to the admissible monomial
 $$
v_h e_{p_1}[i_1]\ldots e_{p_s}[i_s]
 $$
with $h+f_{p_1}+\ldots+f_{p_s}=d$; to choose one monomial from the equivalence class, we assume that~$i_1<-(h,f_{p_1})$.
Consider also the corresponding 
``vacuum'' monomial
 $$
v_h e_{p_1}[(-h,f_1)]e_{p_2}[-(h,f_1)+1]\ldots e_{p_s}[-(h,f_1)+s-1]=v_{d,\varnothing,\varnothing}.
 $$
Note that
 $$
d_1(m_{h,\mathbf{p},\mathbf{i}})-d_1(m_{d,\varnothing,\varnothing})=(i_1+(h,f_1))+\ldots+(i_s+(h,f_1)-s+1).
 $$
The admissible monomial conditions mean that the sequence 
 $$
(i_1+(h,f_1)),\ldots,(i_s+(h,f_1)-s+1)
 $$
is a decreasing sequence of negative integers.  Such sequences can be identified with partitions into negative parts. In addition, the parts of our partitions are coloured (the $p$-data), so the corresponding generating function taking into account all the contributions is $\frac{1}{(q^{-1})_\infty^k}$. It remains to take into account the shift and multiply that by the value of the statistics $d_1$ on the vacuum, which is equal to $-\frac12(h-f_1,h)$. Summing up, we get precisely the formula~\eqref{park-chV} for the character of $\frV$.
\end{proof}

\begin{lemma}
Admissible monomials form a basis of~$\frM$.
\end{lemma}

\begin{proof}
We already know that admissible monomials span~$\frV$, thus there exists a surjection $\frM_{\infty/2}\twoheadrightarrow\frV$, $m_{h,\mathbf{p},\mathbf{i}}\mapsto e_{h,\mathbf{p},\mathbf{i}}$. On the other hand, we know that the characters of~$\frM_{\infty/2}$ and~$\frV$ coincide, so this surjection has to be an isomorphism, and our monomials are linearly independent (and hence form a basis).  
\end{proof}

\begin{theorem}\label{ParkingTh}
Monomials $ve_{p_1}[i_1]e_{p_2}[i_2]\ldots e_{p_n}[i_n]\in\frQ_{n}(k)$ satisfying
\begin{itemize}
 \item[(i)] $i_1=0$ and $i_s\ge0$ for all $s$;
 \item[(ii)] $i_{s+1}\le i_s+1$  for $s=1,\ldots,n-1$; 
 \item[(iii)] if $i_{s+1}=i_s+1$, then $p_s\le p_{s+1}$;
\end{itemize}
form a basis of~$\frQ_{n}(k)$.
\end{theorem}

\begin{proof}
Let us prove that these monomials are linearly independent in~$\frN$. The space~$\frN$ is a universal representation of~$\frB$ with a vector annihilated by all~$e_p[i]$ with~$i>0$, so there exists a surjection $\frM\twoheadrightarrow\frV_0$. The images of our vectors under this surjection are admissible monomials, and thus they are linearly independent. They also span $\frQ_n(k)$, which can be proved using exactly the same reasoning as in Lemma~\ref{park-span}. The theorem is proved. 
\end{proof}

\subsection{Parking functions}

Combinatorial results used in this section (the bijection between parking functions and ``labelled Dyck paths'') can be found in many places; one of the best expositions of relevant properties of parking functions (and their generalisations that we discuss later) can be found in~\cite{Loehr}.

\begin{definition}
A \emph{parking function of length~$n$} is a mapping $$f\colon\{1,\ldots,n\}\to\{1,\ldots,n\}$$ such that for each $k=1,\ldots,n$
 $$
\# f^{-1}(\{1,\ldots,k\})\ge k.
 $$
\end{definition}

Let us encode parking functions using some combinatorial data. Given a parking function~$f$, we rearrange arguments according to their values, and for arguments with the same value~--- in increasing order. Thus we obtain a permutation $\sigma_f$ of~$[n]$. Also, to any parking function we assign the sequence $\{b_k(f)\}_{k=1}^{n}$, where $b_k(f)=\# f^{-1}(\{1,\ldots,k\})$. 

The following lemma is straightforward.

\begin{lemma}
The above correspondence is a bijection between the set of all parking functions of length $n$ and the set of pairs $(\sigma,b)$, where $\sigma$ is a permutation of~$[n]$, and $b$ is a sequence $b_1$, \ldots, $b_n$ of nonnegative integers which satisfy the following conditions:
\begin{itemize}
 \item[(i)] $b_1\le b_2\le \ldots\le b_{n-1}\le b_n=n$;
 \item[(ii)] $b_k\ge k$  for $k=1,\ldots,n-1$;
 \item[(iii)] if $b_{s+1}=b_s$, then $\sigma(s)<\sigma(s+1)$. 
\end{itemize} 
\end{lemma}

\begin{corollary}\label{dyck}
The set of all parking functions of length~$n$ is in one-to-one correspondence with pairs $(\sigma,a)$, where $\sigma$ is a permutation of~$[n]$, and $a$ is a sequence $a_1$, \ldots, $a_n$ of nonnegative integers which satisfy the following conditions:
\begin{itemize}
 \item[(i)] $a_1=0$;
 \item[(ii)] $a_{s+1}\le a_s+1$  for $s=1,\ldots,n-1$;
 \item[(iii)] if $a_{s+1}=a_s+1$, then $\sigma(s)<\sigma(s+1)$. 
\end{itemize}
\end{corollary}

\begin{proof}
A bijection between the $b$-data and the $a$-data is established by the mapping $a_s:=b_{1+n-s}-n-1+s$.
\end{proof}

\begin{remark} 
If we draw the graph of the piecewise constant function defined by the sequence $b$, and decorate this graph using the permutation~$\sigma$, we get a nice pictorial way to represent parking functions. Lattice paths from $(0,0)$ to $(n,n)$ that never go below the diagonal are called \emph{Dyck paths}, so it is more than natural that our type of data is called \emph{labelled Dyck paths}. For example, the parking function $f$ with $f(1)=2$, $f(2)=1$, $f(3)=7$, $f(4)=1$, $f(5)=1$, $f(6)=4$, $f(7)=2$ corresponds to the following picture: 
 $$
\includegraphics[scale=0.8]{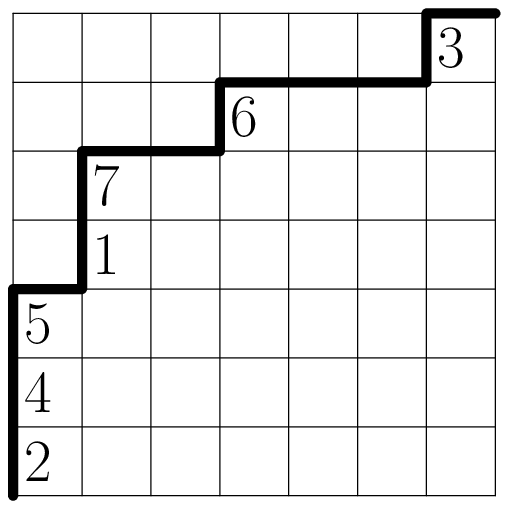}
 $$
\end{remark}

Note that the parking functions condition is a restriction on sizes of preimages, so it is stable under the action of the group~$S_n$ via permutations of the set of arguments. Recall that the parking functions module~$\PF(n)$ is the set-theoretic representation of~$S_n$ arising from the action of the symmetric group on parking functions. For the definition of the Frobenius character $\mathcal{F}_{\PF(n)}$ (used in Theorem~\ref{ParkingCor2}), we refer the reader to~\cite{Mac}, which is the best possible reference on symmetric functions.

\subsection{Proofs}

\emph{Proof of Theorem~\ref{ParkingCor1}.}
\begin{enumerate}
 \item  It is well known (see, for example, \cite{Sta}) that the number of parking functions of length~$[n]$ is equal to $(n+1)^{n-1}$. From Corollary~\ref{dyck} we see that the basis elements of $\frQ_{n}^{multi}(n)$ are in one-to-one correspondence with parking functions, which completes the proof. \qed
 \item Consider the following partial order on the set of admissible monomials that form a basis of $\frQ_{n}^{multi}(n)$: we put 
 $$
e_{0,\mathbf{p},\mathbf{i}}<e_{0,\mathbf{p}',\mathbf{i}'}
 $$
if $\mathbf{i}<\mathbf{i}'$ lexicographically. Take an admissible monomial $e=e_{0,\mathbf{p},\mathbf{i}}$ (which corresponds to a certain parking function~$f$), and some permutation~$\sigma$. The monomial $\sigma(e)$ is either admissible, or can be expressed as a combination of admissible monomials, by applying the procedure that allowed us to prove Lemma~\ref{park-span}. In both cases, it is easy to see that $\sigma(e)$ is equal to the monomial $e'$ corresponding to the parking function $\sigma(f)$ plus a combination of monomials which are larger than $e'$ (and hence larger than~$e$, since the monomials corresponding to $f$ and $\sigma(f)$ the same $\mathbf{i}$-data) . It follows that the matrices of permutations have the same diagonal elements when written in our basis of $\frQ_{n}^{multi}(n)$ and in the natural basis of the parking functions module, hence these two representations have the same character and are isomorphic. \qed
\end{enumerate}

\emph{Proof of Theorem~\ref{ParkingCor2}.} Assume that the representation $V$ of the symmetric group $S_n$ arises from the action on a finite set~$M$ (in the same way as the parking functions module arises from the action on parking functions). Note that if we expand the Frobenius character of~$V$ in the base of monomial symmetric functions, the coefficient of the monomial symmetric function $m_\mu$ is equal to the number of $S_\mu$-orbits in~$M$, where $S_\mu=S_{\mu_1}\times S_{\mu_2}\times\dots\times S_{\mu_k}$ is the corresponding Young subgroup. Indeed, since the base of monomial symmetric functions is dual to the base of the complete symmetric functions, that coefficient is equal to 
\begin{multline*}
\langle h_\mu,\mathcal{F}_V\rangle=\dim\Hom_{S_n}(\Ind_{S_\mu}^{S_n}\mathbbold{1},V)=\\=\dim\Hom_{S_\mu}(\mathbbold{1},\Res_{S_n}^{S_\mu}V)=\dim(V^{S_\mu})=\# M/S_\mu
\end{multline*}
(here $\mathbbold{1}$ denotes the trivial representation of the Young subgroup, and $\langle\cdot,\cdot\rangle$ is the standard inner product on the ring of symmetric functions). The bijection between parking functions and labelled Dyck paths described above can be easily extended to orbits of $S_\mu$. An orbit of $S_\mu$ in the set of parking functions is completely determined by the preference set of the first $\mu_1$ cars, the preference set of the next $\mu_2$ cars etc., so we should just replace the $\sigma$-data define above by the $\widehat{\sigma}$-data, where $\widehat{\sigma}$ is now an ordered multiset, not a permutation: instead of each number, we write down the number of the ``cluster'' containing this number (that is, $1$ for numbers from $1$ to $\mu_1$, $2$ for numbers from $\mu_1+1$ to $\mu_1+\mu_2$ etc.). The data thus obtained is exactly the data that describes admissible monomials.\qed

\section*{Appendix: remarks and open problems}

\subsection*{Catalan numbers}

If we restrict ourselves to some obvious subalgebras of our algebras, we get vector spaces whose dimesions are Catalan numbers. Namely, denote by~$\frA$ the associative algebra with generators $e[i]$ ($i\in\mathbb{Z}$) and relations $$\{e[i],e[j]\}=\{e[i+1],e[j-1]\}$$ for all $i,j$. Consider the right $\frA$-module $\frM$ induced from the trivial representation of the subalgebra $\frA_+$ generated by $e[i]$ for all $i>0$, and denote by~$\frA_{\ge0}$ the linear span $\mathbb{C}\{ e[i]\mid i\ge0\}$. Let~$\frP_{n}$ be the image of the mapping from the $n^\text{th}$ tensor power~$\frA_{\ge0}^{\otimes n}$ to~$\frM$ that maps every monomial~$e[i_1]\otimes e[i_2]\otimes\ldots\otimes e[i_n]$ to~$ve[i_1]e[i_2]\ldots e[i_n]$. 

\begin{theorem}\label{CatalanCor}
The dimension of the space $\frP_{n}$ is equal to the $n^\text{th}$ Catalan number $c_n=\frac{1}{n+1}\binom{2n}{n}$.
\end{theorem}

\begin{proof}
Note that if we restrict ourselves to generators of the first colour, we are dealing with the algebra~$\frA$. The dimension formula follows from the character formula from theorem \ref{ParkingCor2} in the case of the partition~$\mu=(n)$; the coefficient of $m_{n}$, or, in other words, the number of $S_n$-orbits on the set of parking functions, is equal to the $n^\text{th}$ Catalan number.
\end{proof}

\subsection*{Fuss--Catalan numbers and ``higher parking functions''}

Our results admit some rather straightforward generalisations. We formulate them here; the proofs are left to the reader (basically, the main difference in this case is that we put all values of bilinear form on the basis to be equal to the negative of some fixed integer~$m\ge1$; the case $m=1$ was considered in detail in this paper). 
Denote by~$\frB^{(m)}(k)$ the associative algebra with generators $e_p[i]$ ($p=1$, \ldots, $k$, $i\in\mathbb{Z}$) and relations 
 $$
\sum_{i=0}^m (-1)^i\binom{m}{i} [e_p[r-i],e_q[s+i]]_m=0,
 $$
for all $r,s$, $p,q$. 

Consider the right $\frB^{(m)}(k)$-module $\frN^{(m)}$ induced from the trivial representation of the subalgebra $\frB^{(m)}_+(k)$ generated by $e_p[i]$ for all $p=1,\ldots,k$ and all $i>0$. Let~$\frB^{(m)}_{\ge0}(k)=\mathbb{C}\{ e_p[i]\mid i\ge0\}$. Let~$\frQ^{(m)}_{n}(k)$ be the image of the mapping from the $n^\text{th}$ tensor power~$(\frB^{(m)}_{\ge0}(k))^{\otimes n}$ to~$\frN^{(m)}$ that maps the monomial~$e_{p_1}[i_1]\otimes e_{p_2}[i_2]\otimes\ldots\otimes e_{p_n}[i_n]$ to~$ve_{p_1}[i_1]e_{p_2}[i_2]\ldots e_{p_n}[i_n]$. 

\begin{theorem}\label{HigherParkingTh}
Monomials $ve_{p_1}[i_1]e_{p_2}[i_2]\ldots e_{p_n}[i_n]$ satisfying
\begin{itemize}
 \item[(i)] $i_1=0$;
 \item[(ii)] $i_{s+1}\le i_s+m$  for $s=1,\ldots,n-1$;
 \item[(iii)] if $i_{s+1}=i_s+m$, then $p_s\le p_{s+1}$, 
\end{itemize}
form a basis of the vector space~$\frQ^{(m)}_{n}(k)$.
\end{theorem}

\begin{theorem}\label{HigherParkingCor}
Consider a subspace $(\frQ^{(m)}_{n})^{multi}(n)$ of $\frQ^{(m)}_{n}(n)$ spanned by all monomials containing generators of all colours (so that each colour occurs exactly once). The dimension of this subspace is equal to~$(mn+1)^{n-1}$.
\end{theorem}

The corresponding ``higher parking functions'' modules also can be defined as set representations of the symmetric groups. Also, subalgebras corresponding to the rank one sublattices provide generalisations of the Catalan numbers, the so called Fuss--Catalan numbers, as follows.

For $m\ge 1$, consider an associative algebra $\frA^{(m)}$ with generators $e[i]$, $i\in\mathbb{Z}$, and relations 
 $$
\sum_{i=0}^m (-1)^i\binom{m}{i} [e[r-i],e[s+i]]_m=0,
 $$
for all $r,s$, where $[a,b]_m=ab-(-1)^m ba$. Consider the right $\frA^{(m)}$-module $\frM^{(m)}$ induced from the trivial representation of the subalgebra $\frA^{(m)}_+$ generated by $e[i]$ for all $i>0$. Let~$\frA^{(m)}_{\ge0}=\mathbb{C}\{ e[i]\mid i\ge0\}$. Let~$\frP^{(m)}_{n}$ be the image of the mapping from the $n^\text{th}$ tensor power~$(\frA^{(m)}_{\ge0})^{\otimes n}$ to~$\frM^{(m)}$ that maps every monomial~$e[i_1]\otimes e[i_2]\otimes\ldots\otimes e[i_n]$ to~$ve[i_1]e[i_2]\ldots e[i_n]$. 

\begin{theorem}\label{HigherCatalanCor}
The dimension of the space $\frP^{(m)}_{n}$ is equal to the $n^\text{th}$ Fuss--Catalan number $c^{(m)}_n=\frac{1}{mn+1}\binom{(m+1)n}{n}$.
\end{theorem}

\subsection*{Hilbert schemes and rational Cherednik algebras}

In his works on diagonal harmonics~\cite{Ha1,Ha2} Mark Haiman introduced $q,t$ versions of the above numbers whose specialisations at $t=1$ become the $q$-analogues we mentioned. Moreover, in his setting the families of bigraded vector spaces of dimensions $c^{(m)}_n$ and $(mn+1)^{n-1}$ also appear in the most natural way. For example, 
the $c_{n}^{(m)}$-dimensional space from Haiman's work is the space of global sections of the vector bundle  $\xi^{\otimes m}$ on the zero fibre $Z_n$ of the Hilbert scheme ${\rm Hilb}_n(\mathbb{C}^2)$ (here $\xi$ denotes the top exterior power of the tautological bundle on the Hilbert scheme). 

We expect that our vector spaces $\frP_n^{(m)}$ are related to Haiman's spaces. More precisely, we expect that the following conjecture holds.

\begin{conjecture}
There exist some natural filtrations on vector spaces $\frP_{n}^{(m)}$ and mappings of vector spaces 
$$\frP_{n}^{(m_1)}\otimes\frP_{n}^{(m_2)}\to \frP_{n}^{(m_1+m_2)}$$
compatible with filtrations such that on the level of the corresponding graded vector spaces:
\begin{itemize}
\item[(i)] bigraded characters (w.r.t.\,the original grading and the grading obtained from the filtration) coincide with Haiman's bigraded characters;
\item[(ii)] the spaces are identified with the corresponding spaces of global sections and the mappings become taking products of global sections;
\end{itemize}
\end{conjecture}

Let us provide some evidence to support our conjecture. In \cite{Ha2}, Haiman proved that the spaces dual to $c_n$-dimensional spaces from his work are cyclic modules over the abelian Lie algebra $\calL$ with generators of degrees $1,2,\ldots,n,\ldots$ (the same algebra acts on the $c_{n}^{(m)}$-dimensional spaces as well). Our spaces $\frP_n$ (and actually all spaces $\frP_{n}^{(m)}$) are modules over the subalgebra $H_+$ of the Heisenberg algebra spanned by $f[i]$ with $i>0$. Moreover, $\frP_n^*$ is a cyclic module over this algebra. This follows from
\begin{lemma}
Denote by~$\frU_n$ the subspace of~$\calF_{n}$ spanned by all vectors of the form $v_0e[i_1]\ldots e[i_n]$ (where $i_k$ are not necessarily nonnegative). Then $\frU_n^*$ can be naturally identified with the space of rational functions of the type $$\dfrac{P(x_1,\ldots,x_n)}{\prod\limits_{i<j}(x_i-x_j)},$$ where $P$ is a symmetric polynomial. Under this identification, the operator $f[k]$ with $k>0$ acts by multiplication by the $k^\text{th}$ power sum $p_k=x_1^k+\ldots+x_n^k$.
\end{lemma}
(Therefore the module $\frC_{n}^*$ is cyclic, being a quotient of the cyclic module~$\frU_n^*$.)

Gordon and Stafford \cite{GS} related Haiman's spaces to finite-dimensional representations of rational Cherednik algebras. Recall that the rational Cherednik algebras~$\mathbf{H}_{c}$ of type~$A$ have, for the special values~$c=m+1/n$, finite-dimensional representations of dimension~$(mn+1)^{n-1}$ (and their spherical subalgebras have, for the same~$c$, representations whose dimensions are Fuss--Catalan numbers). One of the main result in \cite{GS} (previously a conjecture of Berest, Etingof, and Ginzburg~\cite{BEG}) is that those finite-dimensional representations can be thought of as filtered versions of Haiman's spaces (up to a twist by the sign representation). We expect that our constructions are related to these ones, and hope to address this in more detail elsewhere.

Finally, it would be interesting to understand whether some versions of our vertex algebras can be defined on the spaces of global sections geometrically, in the spirit of beautiful constructions of Grojnowski~\cite{Gro} and Nakajima~\cite{Na}.

\end{document}